\documentclass[a4paper, 11pt]{article}
\usepackage{amsfonts, amsmath, amssymb, amsthm,url,dsfont,mathtools}
\usepackage{graphicx}
\usepackage[english] {babel}
\usepackage{enumitem}
\usepackage{array}

\usepackage{geometry}
\newtheorem{thm}{Theorem}[section]
\newtheorem{cor}[thm]{Corollary}
\newtheorem{lem}[thm]{Lemma}

\newtheorem{prop}[thm]{Proposition}

\theoremstyle{definition}
\newtheorem{cond}[thm]{Condition}

\newtheorem{rem}{Remark}

\DeclareMathOperator{\Var}{Var}

\newcommand{\R}{\mathbb{R}}
\newcommand{\supp}{\textrm{supp}}

\newcommand{\Z}{\mathbb{Z}}
\newcommand{\La}{\mathbb{L}}

\newcommand{\eps}{\varepsilon}

\DeclareMathOperator{\SO}{SO}

\newcommand{\Ha}{\mathcal{H}}

\title{Asymptotic variance of grey-scale surface area estimators}
\author{Anne Marie Svane}

\begin{document}

\maketitle

\begin{abstract}
%
Grey-scale local algorithms have been suggested as a fast way of estimating surface area from grey-scale digital images. Their asymptotic mean has already been described. In this paper, the asymptotic behaviour of the variance is studied in isotropic and sufficiently smooth settings, resulting in a general asymptotic bound. For compact convex sets  with nowhere vanishing Gaussian curvature, the asymptotics can be described more explicitly. As in the case of volume estimators, the variance is decomposed into a lattice sum and an oscillating term of at most the same magnitude.
\end{abstract}

\section{Introduction}
The motivation for this paper comes from digital image analysis. Scientists in e.g.\ materials science and neurobiology are analysing digital output from microscopes and scanners in order to gain geometric information about materials \cite{digital, mecke2}. Common features of interest are volume, surface area, and Euler characteristic, as well as curvature and anisotropy properties. The focus of this paper will be on surface area estimation. Convergent algorithms for surface area are known {\cite{coeurjolly,sun}, but as the amount of output data is typically quite large, there is a need for faster algorithms. 

The simplest model for a digital image is a black-and-white image. If $X\subseteq \R^d$ is the object under study, the set of black pixels is modeled by $X\cap \La$  where $\La$ is a lattice. It is well known that the volume of $X$ can be estimated by $c_\La \cdot \#(X\cap \La)$ where $c_\La $ is the volume of a lattice cell and $\#$ is cardinality. If $\La$ is randomly translated, the mean estimate is exactly the volume.  

Surface area is often estimated in a similar way \cite{digital, lindblad, mecke2}. The idea is to count the number of times each of the $2^{n^d}$ possible $n\times \dots \times n$ configurations of black and white points appear in the image and estimate the surface area by a weighted sum of configuration counts. 
The advantage of these so-called local algorithms is that the computation time is linear in the data amount, see \cite{OM}. However, they are generally biased, even when the resolution tends to infinity \cite{am3,johanna}. 

A more realistic model for a digital image is that of a grey-scale image where we do not observe the indicator function $\mathds{1}_X$ for $X$ itself on $\La$, but rather its convolution $\mathds{1}_X * \rho $ with a point spread function (PSF) $\rho$. In \cite{am4}, local algorithms for grey-scale images are suggested and these are shown to be asymptotically unbiased when the lattice is stationary random and the PSF becomes concentrated near 0. They resemble the volume estimators as they are also given by lattice point counting, but each lattice point must now be weighted according to its grey-value. A simple such algorithm is given by counting the number of lattice points with grey-value belonging to a fixed interval.

So far, not much is known about the precision of these algorithms. Even though the mean converges, the variance may be large. While the convergence of the mean is independent of resolution, a low resolution would intuitively result in a large variance. The purpose of this paper is to study the variance using theory developed for the volume case.

The first part of the paper provides an asymptotic bound on the variance when resolution and PSF changes. It shows that the biggest contribution comes from the resolution. The bound explicitly depends on the algorithm and the underlying PSF.

The asymptotic bound is rather abstract and thus not useful for applications. In the second part of the paper, more explicit formulas for the variance are derived. As in the volume case, this requires strong conditions on the underlying set, namely smoothness, convexity, and nowhere vanishing Gaussian curvature. As in the volume case \cite{matheron}, the variance can be decomposed into a lattice sum depending only on $X$ through its surface area and an oscillating term of at most the same magnitude.

The model for grey-scale images is introduced in Subsection \ref{g-sdef} and local estimators and the known results about their mean are described in Subsection \ref{thealg}. A short recap of some of the known results for volume estimators is included for comparison in Subsection \ref{varknow} before the main results of the paper are described in Subsection \ref{result}. In the following sections, the main results are formally stated and proved. The paper ends with a discussion of the results and a list of open questions.

\section{Set-up and main results}\label{set-up}
\subsection{Grey-scale images}\label{g-sdef}
Consider a blurred image of a compact set $X\subseteq \R^d$, $d\geq 2$. That is, we do not measure $X$ itself, but only an intensity function
\begin{equation*}
\theta^X :=\mathds{1}_X*\rho: \R^d \to [0,1]
\end{equation*}
 where $\rho$ is a PSF. This is assumed to satisfy:
\begin{itemize}
\item $\rho \geq 0$.
\item $\int_{\R^d} \rho(x) dx =1$.
\item $\rho(x)=\rho(|x|)$ for all $x\in \R^d$.
\end{itemize}
The third condition is necessary in order to obtain the asymptotic unbiasedness results of \cite{am4}. The variance is also of interest for more general $\rho$, but we restrict to the radial case in this paper for simplicity.

 We consider the following transformation of $\rho$
\begin{equation*}
\rho_a(x) = a^{-d} \rho(a^{-1}x).
\end{equation*}
The corresponding intensity function becomes
\begin{equation*}
\theta_a^X := \mathds{1}_X*\rho_a.
\end{equation*}
A digital grey-scale image is modeled as the restriction of $\theta_a^X$ to some observation lattice. 

Let $\La \subseteq \R^d$ be a fixed lattice given by $\La = A \Z^d$ for some invertible matrix $A$. The fundamental cell of $\La$ is denoted $C_\La :=A([0,1)^{d})$. The volume of $C_\La$ is $c_\La = \det A$. The dual lattice is $\La^* = A^{-1} \La$. We shall also consider translation and rotation $Q\La_c = Q( \La+c)$ of the lattice by $c\in C_\La$ and $Q\in \SO(d)$. 

A change of resolution corresponds to a scaling of the lattice by a factor $b>0$. Hence we shall generally be working with the observation lattice $bQ\La_c$.
We often assume that the resolution is a function of $a$, i.e.\  $b:=b(a)$. The case $b=a$ is of particular interest, see the discussion in \cite[Sect. 2.1]{am4}.

The intensity function $t \mapsto \theta^{H_u}_a(tu)$ associated to the halfspace $H_u=\{y\in \R^d \mid \langle y,u\rangle \leq 0\}$ plays a special role. Since it is independent of $u\in S^{d-1}$ and $\theta^{H_u}_a(atu) = \theta^{H_u}_1(tu)$, we shall use the notation
\begin{equation*}
\theta^H(t):= \theta^{H_u}_a(atu)
\end{equation*}
for any $a>0$, $u\in S^{d-1}$. 

\subsection{Grey-scale local algorithms}\label{thealg}
First consider an estimator  for the surface area $S(X)$ of the form 
\begin{equation*}
\hat{S}_0(f)^{a,b}(X)=  a^{-1 }b^d\sum_{z\in b\La_c} f\circ \theta_a^X(z)
\end{equation*}
where the weight function $f:[0,1 ] \to \R$ can be any bounded measurable function that is continuous on $[\beta,\omega]\subseteq (0,1) $ with support $\supp f \subseteq [\beta, \omega] $, possibly having discontinuities at $\beta, \omega$.

Assume that $c\in C_{\La}$ is uniform random so that $b\La_c$ is a stationary random lattice. Then the mean estimator is  
\begin{equation}\label{mean0}
E\hat{S}_0(f)^{a,b}(X)=  a^{-1}c_\La^{-1}\int_{\R^d} f\circ \theta_a^X(z) dz.
\end{equation}
It is shown in \cite{am4} that if $X$ is a $C^1$ manifold (or more generally a so-called gentle set, see \cite{rataj}), then
\begin{equation}\label{mean}
\lim_{a \to 0} E\hat{S}_0(f)^{a,b}(X)= c_\La^{-1}  S(X) \int_\R f\circ \theta^{H} (t) dt
\end{equation}
under mild conditions on $\rho$. Equation 
\eqref{mean} is shown in \cite{am4} when $b=a$, but since \eqref{mean0} is independent of $b$, it holds for any function $b(a)$.  

It follows that if 
\begin{equation*}
\alpha_f := \int_\R f\circ \theta^{H}(t) dt \neq 0,
\end{equation*}
then
\begin{equation*}
\hat{S}(f)^{a,b}:= c_\La \alpha_f^{-1} \hat{S}_0(f)^{a,b}
\end{equation*}
is asymptotically unbiased for $a\to 0$. 

The next problem is to describe the variance. An explicit formula is interesting for estimation purposes. But even weaker results may give a hint about the quality of the algorithm. For instance, \eqref{mean} requires nothing of $f$ or $b$, but clearly, if $b(a)$ is large compared to $a$, we expect to see a large variance. Moreover, a good criterion for the choice of weight function $f$ would be that it has small variance.

To study the variance, we will assume that the lattice is also rotated by a uniform random $Q\in \SO(d)$. 
Write $g_a = f\circ \theta_a^X$ and $g_a^-(x)=g_a(-x)$ for simplicity and consider
\begin{align}\label{varcomp}
E\Big(\hat{S}(f)^{a,b}(X)^2\Big){}& =a^{-2}b^{2d}c_\La^2 \alpha_f^{-2}E\bigg(\sum_{z\in bQ\La_c}  g_a(z)\bigg)^2\\\nonumber
&=a^{-2}b^{2d}c_\La^2\alpha_f^{-2} \int_{\SO(d)}\sum_{z_2\in bQ\La} \int_{C_\La} \bigg(\sum_{z_1\in bQ\La_c}  g_a(z_1) g_a(z_1+z_2)\bigg)dcdQ\\ \nonumber
&=a^{-2} b^{d}c_\La\alpha_f^{-2} \sum_{z_2\in b\La} \int_{\SO(d)} g_a * g_a^- (-Qz_2)dQ\\ \nonumber
&=a^{-2} \alpha_f^{-2}\omega_{d}^{-1} \sum_{\xi \in \La^*} \int_{S^{d-1}}|\mathcal{F} (g_a)(b^{-1}|\xi|u) |^2 du.
\end{align}
Here $\mathcal{F} $ denotes the Fourier transform
\begin{equation*}
\mathcal{F} (g_a)(\xi ) = \int_{\R^d} g_a(x) e^{-2\pi i x\cdot \xi} dx.
\end{equation*}
and $\omega_{d} $ is the surface area of $S^{d-1}$.
The last equality in \eqref{varcomp} follows from the Poisson summation formula which applies if $z \mapsto \int_{S^{d-1} } g_a * g_a^- (|z|u)du $ is continuous and the latter sum is convergent \cite[VII, Cor. 1.8]{stein}. 

Since
\begin{align*}
 (a  \alpha_f)^{-1} \mathcal{F} (g_a) (0) = (a \alpha_f)^{-1} \int_{\R^d} g_a(z)dz = E\hat{S}(f)^{a,b}(X),
\end{align*}
the variance is given by
\begin{equation}\label{formelvar}
\Var(\hat{S}(f)^{a,b}(X)) = (a  \alpha_f)^{-2}  \omega_{d}^{-1}\sum_{\xi \in \La^*\backslash \{0\}} \int_{S^{d-1}}|\mathcal{F} (g_a)(b^{-1}|\xi|u) |^2 du.
\end{equation}

\subsection{Known results for volume estimators}\label{varknow}
The volume estimator for black-and-white images mentioned in the introduction
\begin{equation*}
\hat{V}^b(X) = b^d c_\La \cdot \#(X\cap bQ\La_c)
\end{equation*} 
is unbiased. Describing its variance is a classical topic. The case of a ball goes back to \cite{kendall, kendall2} and this was generalised  in \cite{herz,hlawka} to smooth compact convex sets with nowhere vanishing Gaussian curvature. The variance is studied from a statistical viewpoint in \cite{matern, matheron}. For more recent developments, see \cite{brandolini, janacek, kieu}. 

 Replacing $g_a$ with $\mathds{1}_X$ in \eqref{varcomp} shows that the variance is given by
\begin{equation*}
\Var(\hat{V}^b (X))  = \omega_{d}^{-1} \sum_{\xi \in \La^*\backslash \{0\}} \int_{S^{d-1}}|\mathcal{F} (\mathds{1}_X)(b^{-1}|\xi|u) |^2 du.
\end{equation*}
In \cite{brandolini} it is shown for $X$ convex or $C^{\frac{3}{2}}$ that 
\begin{equation*}
 \int_{S^{d-1}}|\mathcal{F} (\mathds{1}_X)(Ru) |^2 du \in O(R^{-d-1})
\end{equation*}
from which it follows that
\begin{equation*}
\Var(\hat{V}^b (X))  \in O(b^{d+1}).
\end{equation*}
In \cite{janacek} these results were used to give a description of the asymptotic variance.

When $X$ is  a smooth compact convex set of nowhere vanishing Gaussian curvature $K$, a more explicit formula is given in \cite{herz,hlawka}. They show that
\begin{equation*}
|\mathcal{F} (\mathds{1}_X)(Ru)|^2 = R^{-d-1} (K(x(u))^{-1} + K(x(-u))^{-1} + Z_u(R)) +O(R^{-d-2})
\end{equation*}
where $x(u)$ is the unique point with normal vector $u$ and $Z_u(R)$ oscillates between $\pm  2(K(x(u))K(x(-u)))^{-\frac{1}{2}}$.
Similarly, 
\begin{equation*}
\int_{S^{d-1}}|\mathcal{F} (\mathds{1}_X)(Ru)|^2 du = 2 S(X) R^{-d-1}(1 + Z(R)) + O(R^{-d-2})
\end{equation*}
with $|Z(R)|\leq 1$. 

To get rid of the oscillating term, it is the idea of \cite{kieu} to consider a set $X$ scaled by a random factor $s\in (0,\infty)$ with continuous density. Then the mean of $Z(R)$ vanishes asymptotically, that is,
\begin{equation*}
\lim_{R\to \infty} R^{d+1} E|\mathcal{F} (\mathds{1}_X)(Ru)|^2 = 2\omega_{d}^{-1}ES(sX).
\end{equation*}
From this, the authors obtain an explicit formula for the asymptotic variance
\begin{equation*}
\lim_{b\to 0} b^{-d-1}\Var(\hat{V}^b (sX))  = 2\omega_{d}^{-1}ES(sX)  \sum_{\xi \in \La^*\backslash \{0\}} |\xi|^{-d-1}.
\end{equation*}
This formula is convenient for applications, since it only requires an estimate for the mean surface area.

For grey-scale images, there are unbiased volume estimators given by
\begin{equation*}
\hat{V}^{a,b}(X) = b^d c_\La \sum_{z\in bQ\La_c} \theta_a^X(z).
\end{equation*} 
The special case where $\rho $ is the indicator function for a sampling figure is considered in \cite{kieu} . The variance is given by
\begin{equation*}
\Var(\hat{V}^{a,b} (X))  = \omega_{d}^{-1} \sum_{\xi \in \La^*\backslash \{0\}} \int_{S^{d-1}}|\mathcal{F} (\rho)(a b^{-1}|\xi|u)|^2 |\mathcal{F} (\mathds{1}_X)(b^{-1}|\xi|u)|^2du .
\end{equation*}
Since $|\mathcal{F} (\rho)(\xi)|\leq 1$, this generally yields a smaller variance.

\subsection{Description of main results}\label{result}
In Section \ref{X} we shall give an asymptotic bound on the Fourier coefficients in \eqref{formelvar} by an argument similar to \cite{brandolini}. This provides an asymptotic bound on the variance. Under suitable conditions on $\rho$ and assuming $f$ to be $C^3$ on all of $(0,1)$ and $X$ to be $C^3$, Theorem \ref{vargeneral} below shows that
\begin{equation}\label{limsupvar}
\limsup_a ab^{-d}\Var(\hat{S}(f)^{a,b}(X)) \leq M_X \frac{\alpha_{|f|}}{\alpha_f^2}\int_\R |(f\circ \theta^H)'(t)|dt
\end{equation}
where $M_X>0$ is a constant depending only on $X$ and $\La$.

When $a=b$ the variance is of order $O(a^{d-1})$ and we shall see that this is  best possible.
However, the convergence rate \eqref{limsupvar} is not best possible for general functions $b(a)$.
In Section \ref{ballsec}, the case where $X$ is the ball $B(R)$ of radius $R$ is investigated further. The strong symmetry allows explicit formulas to be derived. This is used to investigate the convergence rate for $b$ in various regimes. 
It can always be improved when $b \in o(a)$, in fact,
\begin{equation}\label{hurtig}
\limsup_a a^2 b^{-d-1}\Var(\hat{S}(f)^{a,b}(X)) \leq M_X \frac{1}{\alpha^2_f}\bigg(\int_\R |(f\circ \theta^H)'(t)|dt\bigg)^2.
\end{equation}
The best possible rate, however, strongly depends on the smoothness assumptions. When $b^{-1}a\in o(1)$, the precise convergence rate is not known, but we shall see that it cannot be faster than $a^{-2}b^{d-1}$.

%
 %
The proof of the central Lemma \ref{generalF} yields an approximation formula for the Fourier coefficients. 
When $X$ is smooth and convex with nowhere vanishing Gaussian curvature, this, together with theory developed for the volume case, is applied in Corollary \ref{detcor} below to describe the asymptotic variance as 
\begin{align*}
 \Var(\hat{S}(f)^{a,a}(X)) ={}&  2a^{d-1}  \omega_{d}^{-1}\alpha_f^{-2} S(X) \bigg(\sum_{\xi \in \La^*\backslash \{0\}}\big| \mathcal{F}(f\circ \theta^H)(|\xi|) \big|^2|\xi|^{-d+1} + Z(a)\bigg) \\&+o(a^{d-1})
\end{align*}
where $Z(a)$ is in general an oscillating term depending on $X$ and satisfying
\begin{equation*}
\limsup_a \pm Z(a) \leq \sum_{\xi \in \La^*\backslash \{0\}} \big| \mathcal{F}(f\circ \theta^H)(|\xi|) \big|^2|\xi|^{-d+1}.
\end{equation*}

Assuming that $X$ also has a random radius $s$ with smooth, compactly supported density, $\hat{S}(f)^{a,a}(sX)$ is an asymptotically unbiased estimator for the mean surface area. As in the volume case, the oscillating term vanishes asymptotically in the mean, that is,
\begin{equation*}
\lim_{a\to 0} a^{-d+1}\Var(\hat{S}(f)^{a,a}(X)) =  2 \omega_{d}^{-1}\alpha_f^{-2} ES(sX) \sum_{\xi \in \La^*\backslash \{0\}} \big| \mathcal{F}(f\circ \theta^H)(|\xi|) \big|^2|\xi|^{-d+1}, 
\end{equation*}
see Theorem \ref{randomdet} below. Again, this expression only depends on $X$ through its surface area. The remaining lattice sum depends only on the chosen $f$ and the underlying PSF $\rho$ and can in principle be computed once the function $\theta^H$ is known.  

\section{Prerequisites}
In this section we introduce some more notation and assumptions that will be used throughout and prove a couple of technical lemmas about grey-scale images. 

\subsection{General assumptions and notation}
We will assume that the object $X $ we observe is a compact manifold with $C^3$ boundary. In particular, it allows a tubular neighborhood $T^r$ of radius $r$. 
We let $\xi_{\partial X} : T^r \to \partial X$ be projection onto the boundary. 


The function $\theta^H$ is always decreasing. 
Suppose it is differentiable. If
\begin{equation}\label{thetacond}
\frac{d}{dt}\theta^H(t)<0 \text{ whenever } \theta^H(t) \in (0,1),
\end{equation}
then $\theta^H$ has an inverse defined on $(0,1)$ which we denote by $\varphi: (0,1) \to \R $.

We collect the assumptions we shall make on $\rho$ for later reference. They ensure in particular that $\theta^H$ is differentiable.
\begin{cond}\label{cond1}
$\rho$ is  $C^2$, has compact support, and satisfies \eqref{thetacond} and the conditions of Section \ref{g-sdef}.
\end{cond}

\begin{cond}\label{cond2}
$\rho$ is  $C^2$, satisfies \eqref{thetacond} and the conditions of Section \ref{g-sdef}, and  for some $s>d$
\begin{equation*}
\rho(x), | \nabla \rho (x) |,\bigg|\frac{\partial^2\rho}{\partial x_i \partial x_j}(x)\bigg| \in O(|x|^{-s}). 
\end{equation*}
\end{cond}
The reasons for these assumptions will become clear in the following subsection. Note in particular that Condition \ref{cond2} is satisfied when $\rho$ is the Gaussian.

For short we shall write $g_a = f\circ \theta_a^X$. We choose $D>0$ such that $\supp \rho \subseteq B(D)$ in the case of compact support, and otherwise such that
\begin{equation*}
\int_{B(D)} \rho(z) dz \geq 1-\beta , \omega. 
\end{equation*}
Then $\supp g_a \subseteq T^{aD} \subseteq T^r$ for all $a$ sufficiently small. 

Given $x\in \partial X$, we let $H_x$ denote the supporting halfspace $x+H_{n(x)}$ where $n(x)$ is the unique outward pointing normal vector. Observe that $\theta^{H_x}_a(x+atn)=\theta^H(t)$. This explains why $\theta^H$ shows up in the asymptotic mean \eqref{mean}: It comes from approximating $X$ locally by its tangent halfspace.

For $x\in \partial X$, we write 
\begin{align*}
t_+(x,a){}&=\sup \{t\in [-aD,aD] \mid \theta_a^X(x+tn(x)) \geq \beta \}\\
t_-(x,a){}&=\inf \{t\in [-aD,aD] \mid \theta_a^X(x+tn(x)) \leq \omega\}.
\end{align*}
As we shall see below, our assumptions ensure that  for $a$ sufficiently small, $t_+(x,a)$ and $t_-(x,a)$ are the unique $t\in [-aD,aD]$ with the properties $\theta_a^X(x+tn(x))=\beta$ and $\theta_a^X(x+tn(x))=\omega$, respectively.

\subsection{Some lemmas about grey-scale images}
We now prove some technical lemmas that we need later.
\begin{lem}\label{diff}
Let $\alpha $ be a multiindex.
Suppose $f$ is $C^{|\alpha|+1}$ on $[\beta, \omega]$ and $\rho $ is $C^{|\alpha|}$ with compact support. 
There is a constant $M>0$ such that for all $a$ sufficiently small and all with $ \theta_a^X(y),\theta_a^{H_{\xi_{\partial X}(y)}}(y)\in [\beta,\omega]$, 
\begin{align*}
\bigg|\frac{\partial^{|\alpha|}}{\partial x^\alpha } f\circ \theta_a^X(y) - \frac{\partial^{|\alpha|}}{\partial x^\alpha }\Big(x\mapsto f\circ \theta^{H_{\xi_{\partial X}(y)}}_a\Big)(y) \bigg| \leq M a^{1-|\alpha|}.
\end{align*}

If $\rho$ does not have compact support, the above holds with $a^{\frac{s-d}{s+1}-|\alpha|}$ on the right hand side if
\begin{equation*}
\bigg|\frac{\partial^{|\gamma|}}{\partial x^{\gamma} } \rho(x) \bigg| \in O(|x|^{-s})
\end{equation*}
for all $|\gamma|\leq |\alpha|$. 
\end{lem}

\begin{proof}
Writing $x=\xi_{\partial X}(y)$, we compute
\begin{align*}
\MoveEqLeft  \Big| f\circ \theta_a^X(y) -   f\circ \theta^{H_x}_a(y) \Big|
\leq \sup |f'|\Big|(\mathds{1}_X -  \mathds{1}_{H_x}) * \rho_a (y) \Big|\\
&\leq a^{-d} \sup |f'| \sup| \rho| \lambda \big((X \Delta H_x) \cap ([x-rn(x),x+rn(x)]\oplus B(aD))\big)\\
&\leq Ma
\end{align*}
where $\Delta $ denotes the symmetric difference, $\lambda$ is Lebesgue measure, $[\cdot,\cdot]$ is the line segment and $\oplus $ is the Minkowski sum. 

The case of higher order derivatives follows similarly, using the fact that 
\begin{align*}
\frac{\partial}{\partial x_i } (\mathds{1}_X * \rho_a) (y) = a^{-1} \mathds{1}_X * \bigg( \frac{\partial}{\partial x_i } \rho \bigg)_a(y).
\end{align*}

%

In the non-compact case, for $R< r-aD$ 
\begin{align*}
|(\mathds{1}_X -  \mathds{1}_{H_x}) * \rho_a(y)| 
&\leq M a^{-d} R^{d+1} \sup{|\rho |} + \int_{|x|> a^{-1}R} |x|^{-s} dx\\
&\leq M' a^{\frac{s-d}{s+1}}
\end{align*}
where the last inequality follows by choosing $R=a^{\frac{s}{s+1}}$.
\end{proof}

We shall also need the following lemma, see \cite[Eq. (7.6)]{am4} for the first case and \cite[Lem. 7.1 and 7.2]{am4} for the second case.

\begin{lem}\label{tlemma}
Let $\rho$ satisfy Condition \ref{cond1}.
Then there is a constant $M$ uniform in $x\in \partial X$ such that for $a$ sufficiently large
\begin{equation*}
|t_+(x,a) - a\varphi(\beta)| , |t_-(x,a)- a\varphi(\omega) | \leq Ma^2.
\end{equation*}
When $\rho$ satisfies Condition \ref{cond2}, the same holds with $Ma^{\frac{s-d}{s+1}+1}$ on the right hand side.
\end{lem}

\begin{cor} \label{voks}
Let $\rho$ satisfy Condition \ref{cond2}.
Then 
\begin{align*}
\sup \bigg\{ \frac{d}{dt}\theta^X_a(x+tn) \mid x\in \partial X , t \in [t_-(x,a),t_+(x,a)] \bigg\} < 0
\end{align*}
for all $a$ sufficiently small. In particular, $t\mapsto \theta^X_a(x+tn)$ is strictly decreasing on $[t_-(x,a),t_+(x,a)] $.
\end{cor}

\begin{lem}
Suppose $\rho$ satisfies Condition  \ref{cond2} and $f$ is $C^1$ on $[\beta, \omega]$. Then $g_a * g_a^-$ is continuous. 
\end{lem}

\begin{proof}
The set $E$ where $g_a$ is not continuous is $(\theta_a^X)^{-1}(\{\beta,\omega\}) \subseteq T^{aD}$ which is a compact set of measure zero by Corollary \ref{voks}. 

It follows that 
\begin{align*}
\MoveEqLeft |g_a * g_a^-(x_1)-g_a * g_a^-(x_2) | =\bigg|\int_{\R^d} g_a(z)(g_a(z+x_1)-g_a(z+x_2) )dz\bigg|\\
 \leq {}& M_1 \sup |f|^2 \lambda(E \oplus B(|x_1-x_2|)) + M_2 \sup|f| \sup|f'|\sup|\nabla \theta^X_a||x_1-x_2| \lambda (T^{aD})
\end{align*}
which goes to 0 for $|x_1-x_2| \to 0$ by monotone convergence.
\end{proof}

\section{Asymptotic bound on the variance}\label{X}
In this section we shall obtain the following general bound on the variance:

\begin{thm}\label{vargeneral}
Assume $X$ is a compact manifold with $C^3$ boundary, $f$ is $C^3$ on $(0,1)$ with compact support, and $\rho$ satisfies Condition \ref{cond1}. 
Then there is a constant $M>0$ depending only on $X$ and $\La$ such that for all $a$ and $b$ small and $R$ large
\begin{equation}\label{ulig}
\Var(\hat{S}(f)^{a,b}(R X)) \leq  a^{-1}b^d R^{d-1} M\frac{\alpha_{|f|}}{\alpha_f^2}\int |(f\circ \theta^H)'(t)|dt + O\Big(a^{-\frac{1}{2}}b^d R^{d-\frac{3}{2}}\Big) 
\end{equation}
where the $O$-term is allowed to depend on $\rho$ and $f$. 

If $\rho $ satisfies Condition \ref{cond2} with $s>2d+1$, \eqref{ulig} holds with $O(a^{-k}b^{d}R^{d+k-2})$ on the right hand side for $ 2- 2\frac{s-d}{s+1} <  k < 1$.
\end{thm}

The proof will follow from the following lemma:
\begin{lem}\label{generalF}
Assume that $X$ is $C^3$, $f$ is $C^3$ on $(0,1)$, and $\rho $ satisfies Condition \ref{cond1}. Then there are constants $M_1,M_2>0$ depending only on $X$ such that for all $R$ large and $a$ small,
\begin{align}\label{Fbound}
\int_{S^{d-1}}|\mathcal{F}(f\circ \theta_a^X)(Ru)|^2 du \leq \begin{dcases}
M_1  R^{-d-1} \bigg(\int|(f\circ \theta^H )' (t)|dt\bigg)^2+O\big( a R^{-d-\frac{1}{2}}\big)\\
M_2 a^2 R^{-d+1} \bigg(\int |f\circ \theta^H (t)| dt\bigg)^2  + O\big( a^3 R^{-d+\frac{3}{2}}\big).
\end{dcases}
\end{align}
%
%
%
%

If $\rho $ only satisfies Condition \ref{cond2}, \eqref{Fbound} holds but with $O( a^{-2+(\frac{s-d}{s+1}+1)\eps}R^{-d-2+\eps})$ in the first inequality and $O( a^{(\frac{s-d}{s+1}+1)\eps}R^{-d+\eps})$ in the second inequality where $\frac{2(s+1)}{2s-d+1} <  \eps < 2$.
\end{lem}

The proof essentially follows \cite{brandolini}.
First note that partial integration in the $u$-coordinate in the inner integral yields:
\begin{align}\nonumber
\int_{S^{d-1}}|\mathcal{F}(f\circ \theta_a^X)(Ru)|^2 du{}& = \int_{S^{d-1}}\bigg|\int g_a(x)e^{-2\pi i Rx \cdot u} dx\bigg|^2 du\\
&=\frac{1}{(2\pi R)^2} \int_{S^{d-1}}\bigg|\int \nabla_u g_a(x) e^{2\pi i R x \cdot u} dx\bigg|^2 du.\label{intpar}
\end{align}
Here $\nabla_u f(x)$ denotes the directional derivative of $f$ in direction $u$ evaluated at $x$. This will also sometimes be written $\nabla_u f(x)=\nabla f (x) \cdot u$ where $\nabla f$ is the gradient. The two viewpoints on the integral give rise to the two inequalities.

In \cite{brandolini}, the Fourier integral for $\mathds{1}_X$ is converted to a boundary integral via the divergence theorem. As $f\circ \theta_a^X$ does not live on $X$ but on a neighbourhood of the boundary, we apply instead the Weyl tube formula \cite{weyl}. This states that for a compact manifold $X$ with $C^2$ boundary and a bounded measurable function $g$ living on $T^r$,
\begin{equation*}
\int_{\R^d} g(x) dx = \sum_{m=0}^{d-1} \int_{\partial X} \int_{-r}^{r}t^m g(x+tn) s_m(x) dt \sigma (dx)
\end{equation*}
where $s_m$ is the $m$'th symmetric polynomial in the principal curvatures and $\sigma$ is the surface area measure on $\partial X$. Note that  $s_m$ is $C^1$ under the $C^3$ assumption. 

\begin{proof}
We focus on the case of Condition \ref{cond1}. The second case is similar, only the bounds are slightly different.

As in \cite{brandolini}, choose a covering of $\partial X$ by open sets $X_j$ for which there is a $\xi_j \in S^{d-1}$ so that for all $x,y\in X_j$, the angle between $x-y$ and $ \xi_j $ is at least $\frac{7\pi}{16}$. Choose a smooth partition of unity subordinate to this covering and extend radially to a smooth partition of unity $\{\varphi_j\}$ on $T^r$ by composing with $\xi_{\partial X}$. Hence it is enough to show that for all $m$ and $j$ 
\begin{align*}
&\int_{S^{d-1}}\bigg| \int_{\partial X_j} \int_{-r}^{r} t^m g_a(x+tn)e^{2\pi iR (x+tn) \cdot u} s_m(x) \varphi_j(x) dt \sigma (dx)\bigg|^2du \\
&\quad \leq M_1 a^2 R^{-d+1} \bigg(\int |f\circ \theta^H(dt)|dt\bigg)^2  + O(a^3R^{-d+\frac{3}{2}})\\
&\int_{S^{d-1}}\bigg|\int_{\partial X_j} \int_{-r}^{r} t^m \nabla_u g_a(x+tn) e^{2\pi iR (x+tn) \cdot u} s_m(x) \varphi_j(x)  dt \sigma(dx) \bigg|^2 du \\
&\quad \leq M_2  R^{-d+1} \bigg(\int |(f\circ \theta^H)'(t)|dt\bigg)^2+ O(aR^{-d+\frac{3}{2}}).
\end{align*}
The proofs are essentially the same, hence we shall only give the arguments below for the second, slightly more complicated, inequality. Observe that by linearity of $u\mapsto \nabla_u g_a(x+tn)$, the Minkowski inequality allows us to replace $\nabla_u g_a(x+tn)$ by $\nabla_{u_0} g_a(x+tn)$ for some fixed $u_0$.

By rotating the picture, we may assume that $\xi_j=e_d$ is the $d$th standard basis vector.
Let $\psi $ be a smooth function on $S^{d-1}$ which is 1 on the spherical caps $|\langle e_d, u \rangle|\geq \cos(\frac{\pi}{4})$ and supported on the slightly larger caps $|\langle e_d, u \rangle|\geq  \cos (\frac{3\pi}{8})$.

Consider first
\begin{align*}
\int_{S^{d-1}}\bigg|\int_{\partial X_j} \int_{-r}^{r} t^m \nabla_{u_0} g_a(x+tn)e^{2\pi iR (x+tn) \cdot u} s_m(x) \varphi_j (x) \psi(u) dt \sigma(dx) \bigg|^2 du .
\end{align*}
By the Minkowski integral inequality, this is bounded by
\begin{align*}
\MoveEqLeft \bigg( \int_{-aD}^{aD} \bigg(\int_{S^{d-1}}\bigg|\int_{\partial X_j}  t^m \nabla_{u_0} g_a(x+tn)e^{2\pi i R(x+tn) \cdot u}  \psi (u)  \mu_{m,j} (dx) \bigg|^2 du \bigg)^{\frac{1}{2}}dt\bigg)^2\\
& \leq (aD)^m \bigg(
\int_{-aD}^{aD}  \bigg(\int_{S^{d-1}}\int_{\partial X_j} \int_{\partial X_j}  \nabla_{u_0} g_a(x+tn) \nabla_{u_0} g_a(y+tn) \\
&\quad \times e^{2\pi i R(x-y+t(n(x)-n(y))) \cdot u}  \psi (u)  \mu_{m,j}(dx) \mu_{m,j} (dy)du \bigg)^{\frac{1}{2}}dt\bigg)^2
\end{align*}
where $\mu_{m,j} $ is short for $\varphi_j s_m \sigma$. 

Given $x\neq y \in X_j$, write $w=\frac{x-y+t(n(x)-n(y))}{|x-y+t(n(x)-n(y))|}$. Identifying $S^{d-2}$ with $w^\perp \cap S^{d-1}$, parametrize $S^{d-1}$ by $u: [-1,1] \times S^{d-2}\ \to S^{d-1}$ where $u(s,v) = s w + \sqrt{1-s^2}v$.
This has smooth Jacobian determinant $J(s,v)$ away from $s=\pm 1$. 

Since $n$ is $C^2$, there is a $C>0$ such that $|n(x)-n(y)|\leq C|x-y|$ for all $x,y\in X_j$. Thus for $|t|\leq aD$ and $a$ sufficiently small, 
\begin{equation}\label{vurdering}
|x-y+t(n(x)-n(y))|\geq (1-aDC)|x-y| \geq (1-\delta)|x-y|.
\end{equation}
For $u \in \supp \psi $ and $\delta$ chosen sufficiently small, it follows that there is a $\delta' >0$ such that
\begin{equation*}
|s|= |w \cdot u| \leq (1-\delta)^{-1}|\cos(\tfrac{\pi}{16}) + aDC | \leq  1-\delta'.
\end{equation*}
for $a$ sufficiently small independently of $x,y$. Hence partial integration may be applied  $d$ times in the $s$-coordinate for $a$ sufficiently small:
\begin{align}\nonumber
\MoveEqLeft\int_{-aD}^{aD}  \bigg(\int_{\partial X_j} \int_{\partial X_j}  \int_{S^{d-2}} \int_{-1+\delta}^{1-\delta}  e^{2\pi iR |x-y+t(n(x)-n(y))| s } J(s,v) \psi (u(s,v)) ds dv \\ \nonumber
& \quad \times \nabla_{u_0} g_a(y+tn) \nabla_{u_0} g_a(x+tn)\mu_{m,j}(dx) \mu_{m,j} (dy)\bigg)^{\frac{1}{2}}dt\\ \label{etintegral}
& \leq M_1\int_{-aD}^{aD} \sup |\nabla  g_a |(t)   \\
& \quad \times \bigg( \int_{\partial X_j} \int_{\partial X_j} (( R |x-y+t(n(x)-n(y))|)^{-d} \wedge 1 ) |\mu_{m,j}|(dx) |\mu_{m,j}| (dy) \bigg)^{\frac{1}{2}} dt\nonumber
\end{align}
where $\sup |\nabla g_a | (t)=\sup \{ |\nabla g_a (x+tn) | \mid x\in X_j \}$.

Using \eqref{vurdering} again, the fact that $|\mu_{m,j}| \leq c \sigma $ and has compact support, and a parametrization of $X_j$ as a graph over a hyperplane, shows that \eqref{etintegral}  is bounded by
\begin{align*}
\MoveEqLeft M_2\int_{-aD}^{aD}  \sup |\nabla  g_a |(t) \bigg(\int_{\partial X_j} \int_{\partial X_j} (( R |x-y|)^{-d} \wedge 1 ) \sigma (dx) \sigma (dy) \bigg)^{\frac{1}{2}}dt\\
& \leq  M_3 R^{\frac{d-1}{2}}a \int_{-D}^{D}  \sup |\nabla  g_a |(at)  dt.
\end{align*}
Finally, there is a constant $M$ such that $| \nabla g_a (x+  atn)- a^{-1}\nabla (f\circ \theta^H)(t)|\leq M$ uniformly in $t$ and $x$ by Lemma \ref{diff}. This yields the required bound.

It remains to bound
\begin{align*}
\bigg(\int_{-aD}^{aD} \bigg(\int_{S^{d-1}}\bigg|\int_{\partial X_j} \nabla_{u_0} g_a(x+tn)e^{2\pi i R (x+tn) \cdot u} (1- \psi (u))  \mu_{m,j}(dx) \bigg|^2 du \bigg)^{\frac{1}{2}}dt \bigg)^2.
\end{align*}
Cover the support of $1- \psi(u)$ by coordinate neighborhoods that are rotations of the following:  $(\theta,v) \in (-\frac{3\pi}{8},\frac{3\pi}{8} )\times S^{d-2}$ corresponds to $Q_\theta(v)$ where $v\in e_d^\perp \cap S^{d-1}$ and $Q_\theta$ is rotation by the angle $\theta$ in the $\{e_1,e_d\}$-plane. This has smooth Jacobian determinant $J_1(\theta,v)$. We take the cap with the halfspace $\langle e_1 , x \rangle \geq \eps_1 >0$. Choosing a smooth partition of unity $\eta_l(\theta, v)$ with respect to this covering and the caps $|\langle u,e_d\rangle| \geq \cos( \frac{\pi}{4})$ and rotating the picture again, we may assume this is exactly the coordinate system.

For $t$ fixed, let $X_j^t= \{x+tn(x)\mid x\in X_j \}$. Note that this is a $C^2$ surface with normal vector $n(x)$ at $x+tn(x)$. Hence it can locally be written as the graph of $\tilde{A}_\theta^t : Q_\theta e_d^\perp \to X_j^t$ over $Q_\theta e_d^\perp$.

By the inverse function theorem, the family $A_\theta^t(y) =\tilde{A}_\theta^t(Q_{\theta}y)$ for $(y,t,\theta) \in e_d^\perp \times [-r,r] \times (-\frac{3\pi}{4},\frac{3\pi}{4})$ defines a $C^2$ local diffeomorphism onto $T^r \times (-\frac{3\pi}{4},\frac{3\pi}{4})$. 
The assumption 
\begin{equation*}
 (x-y) \cdot Q_\theta e_d < \cos(\tfrac{\pi}{16})|x-y|<|x-y|
\end{equation*}
 for all $x,y \in X_j$ ensures injectivity on $e_d^\perp \times \{0\} \times (-\frac{3\pi}{4},\frac{3\pi}{4})$ and hence by compactness there is a small $r$ such that for every $(\theta,v) \in \supp (1-\psi) $, $T^r$ is globally parametrized by $(y,t) \mapsto A_\theta^t(y)$.


For $(\theta,t)$ fixed, we thus assume $X_j^t $ parametrized by $A_\theta^t(y)$ where $y\in \R^{d-1} \cong e_d^\perp$. By the above, the determinant $J_\theta^t(y)$ of the Jacobian  of $A_\theta^t$ is $C^1$ in $(y,\theta, t)$. Thus our integral becomes
\begin{align*}
 \bigg(\int_{-aD}^{aD}  \bigg(\int_{-\frac{3\pi}{4}}^{\frac{3\pi}{4}}\int_{S^{d-2}}\bigg|\int_{\R^{d-1}} \nabla_{u_0} g_a(A_t^\theta(y))e^{2\pi i R  y \cdot v } J_\theta^t(y) s_m(\xi_{\partial X}(A_\theta^t(y)))\\
 \varphi_j(A_\theta^t(y)) dy \bigg|^2  J_1(\theta,v)^2 \eta_l(\theta,v)^2 (1- \psi(\theta,v))^2  dv d\theta\bigg)^{\frac{1}{2}} dt\bigg)^2.
\end{align*}

Let 
\begin{equation*}
B_{\theta}^t (y)= J_\theta^t(y) s_m(\xi_{\partial X}(A_\theta^t(y))) \varphi_j(A_\theta^t(y))
\end{equation*}
and observe that this is differentiable in $y$ with derivative continuous in $(t,\theta,u)$ for $t$ small enough.

Write $y=(y_1,y')$ and $v=(v_1,v')$. As in \cite{brandolini} we introduce the difference operator in the first variable
\begin{equation*}
\Delta^1_h f(y) = f(y_1+h,y') - f(y)
\end{equation*}
and observe that $e^{2\pi i R  y \cdot v }=(e^{ i v_1} -1)^{-1}\Delta_{\frac{1}{2\pi R}}^1 e^{2\pi i R y \cdot v }$.
Hence our integral becomes 
\begin{align*}
\MoveEqLeft \bigg( \int_{-aD}^{aD} \bigg(\int_{-\frac{3\pi}{4}}^{\frac{3\pi}{4}}\int_{S^{d-2}}\bigg|\int_\R \frac{e^{2\pi iRy_1v_1}}{e^{ i v_1} -1}\int_{\R^{d-2}} \Delta_{\frac{-1}{2\pi R}}^1( \nabla_{u_0}  g_a(A_\theta^t(y)) B^t_\theta (y)) e^{2\pi i R y' \cdot v' }  dy'dy_1 \bigg|^2 \\
& \times J_1(\theta,v)^2 \eta_l(\theta,v)^2 (1- \psi(\theta,v))^2   dv d\theta\bigg)^{\frac{1}{2}} dt\bigg)^2.
\end{align*}
Since $v_1 > \eps_1 $ by assumption, the Minkowski integral inequality shows that this integral is bounded by
\begin{align*}
\MoveEqLeft M_1 \bigg( \int_{-aD}^{aD} \int_\R \bigg(\int_{-\frac{3\pi}{4}}^{\frac{3\pi}{4}}\int_{S^{d-2}}\bigg|\int_{\R^{d-2}} \Delta_{\frac{-1}{2\pi R}}^1(  \nabla_{u_0}g_a(A_\theta^t(y)) B^t_{\theta} (y)) e^{2\pi i R y' \cdot v' }  dy' \bigg|^2  \\
& \times J_1(\theta,v)^2 \eta_l(\theta,v)^2 (1- \psi (\theta,v))^2  dv d\theta\bigg)^{\frac{1}{2}} dy_1  dt\bigg)^2.
\end{align*}

The argument now follows \cite{brandolini}: Partial integration in $v'$ shows that up to a constant, the integral inside the square root is bounded by
\begin{align*}
\MoveEqLeft \int_{-\frac{3\pi}{4}}^{\frac{3\pi}{4}}\int_{S^{d-2}}\int_{\R^{d-2}} \int_{\R^{d-2}}((R|y'-z'|)^{-d+1}\wedge 1) |\Delta_{\frac{-1}{2\pi R}}^1(  \nabla_{u_0}g_a(A_\theta^t(y_1,y')) B_{\theta}^t (y_1,y')) |\\
&\times |\Delta_{\frac{-1}{2\pi R}}^1(  \nabla_{u_0}g_a(A_\theta^t(y_1,z')) B_{\theta}^t (y_1,z')) | dy' dz'  dv d\theta\\
\leq {}& M_2R^{d-2}\int_{-\frac{3\pi}{4}}^{\frac{3\pi}{4}}\int_{\R^{d-2}} \mathcal{M}'(|\Delta_{\frac{-1}{2\pi R}}^1(  \nabla_{u_0}g_a(A_\theta^t(y)) B_{\theta}^t (y)) |) \\&\times |\Delta_{\frac{-1}{2\pi R}}^1(  \nabla_{u_0}g_a(A_\theta^t(y)) B_{\theta}^t (y)) | dy'  d\theta
\end{align*}
where $\mathcal{M}'$ is the Hardy-Littlewood maximum and the inequality follows from \cite[3.2, Thm. 2]{stein2}. By the Cauchy-Schwarz inequality and boundedness of $\mathcal{M}'$ on $L^2$ \cite[1.1, Thm. 1]{stein2}, the full integral is bounded by 
\begin{align*}
\MoveEqLeft M_3 R^{-d+2} \bigg( \int_{-aD}^{aD}  \int_\R \bigg(\int_{-\frac{3\pi}{4}}^{\frac{3\pi}{4}}\int_{\R^{d-2}} |\Delta_{\frac{-1}{2\pi R}}^1 (  \nabla_{u_0}g_a(A_\theta^t(y)) B_{\theta}^t (y))|^2   dy' d\theta\bigg)^{\frac{1}{2}}dy_1  dt\bigg)^2.
\end{align*}

Finally, to bound the integrand, recall that it is compactly supported in $y$ and
\begin{equation*}
\Delta_{h}^1 (fg)(y) = f(y) \Delta_{h}^1 g(y) + g(\Delta_{h}^1(y)) \Delta_{h}^1 f(y).
\end{equation*}
%
We have that
\begin{align*}
|\nabla_{u_0}g_a(A^t_\theta(y)) \Delta_{\frac{-1}{2\pi R}}^1 (B_{\theta}^t (A_\theta^t(y)))| &\leq M_4 \sup_{y\in X_j^t} |\nabla g_a(A_\theta^t(y)) | R^{-1},\\
 |\sup_{y\in X_j^t} |\nabla g_a(A_\theta^t(y)) | - a^{-1} |(f\circ \theta^H)'(a^{-1}t)||&\leq  M_5.
\end{align*}
Here $M_5$ is uniform but may depend on $\rho $ and $f$ by Lemma \ref{diff}. 

Moreover, 
\begin{align*}
|\Delta_{\frac{-1}{2\pi R}}^1 \nabla_{u_0} g_a(A_\theta^t(y))| \leq \begin{cases} M_6  a^{-2} R^{-1}\\
M_7 + M_8 a^{-1}R^{-1}| (f\circ \theta^H )'(a^{-1}t)|)\\
\end{cases}
\end{align*}
where $M_6$ and $M_7$ may depend on $f$ and $\rho$.
The first inequality uses the mean value theorem and  Lemma \ref{diff}. The second uses the fact that 
\begin{equation*}
 |\nabla_{u_0} g_a(x+tn)- \nabla_{u_0} f\circ \theta^{H_x}_a(x+tn)|
\end{equation*}
 is bounded and 
\begin{align*}
|\nabla_{u_0} f\circ \theta_a^{H_x}(x+tn)-\nabla_{u_0} f\circ \theta_a^{H_y} (y+tn)|{}&\leq a^{-1}| (f\circ \theta^H )'(a^{-1}t)||n(x)-n(y)| \\
&\leq a^{-1}C | (f\circ \theta^H )'(a^{-1}t)| |x-y |\\
& \leq  a^{-1} C| (f\circ \theta^H )'(a^{-1}t)|({R}^{-1} + 2aD).
\end{align*}
 Hence
\begin{align}\label{estimat}
|\Delta_{\frac{-1}{2\pi R}}^1 \nabla_{u_0} g_a(A_\theta^t(y))|^2 \leq M_9 a^{-2}R^{-2}|(f\circ \theta^H )'(a^{-1}t)|^2+  M_{10} a^{-1} R^{-\frac{1}{2}}.
\end{align} 
The latter yields the $O(aR^{-d-\frac{1}{2}})$ term.

The case of $g_a$ is exactly similar, only the last step is slightly simpler because $x \mapsto f\circ \theta_a^{H_x} (x+tn)$ is constant.
The case when $\rho $ satisfies Condition \ref{cond2} is also similar, now using $x^2=x^{2-\eps}x^{\eps}$ in \eqref{estimat}.
\end{proof}

\begin{proof}[Proof of Theorem \ref{vargeneral}]
Observe first that $\mathcal{F}(f\circ \theta_{a}^{RX})(\xi)=R^{d}\mathcal{F}(f\circ \theta_{aR^{-1}}^{X})(R\xi)$.

It follows from the second estimate in Lemma \ref{generalF} that the Poisson summation formula applies. 
Using both parts of \eqref{Fbound} and the estimates
\begin{align*}
\sum_{\substack{\xi \in \La^*\backslash \{0\},\\ |\xi| < C}} |\xi|^{-d+1}\leq K_1^\La   \int_{|x|<C} |x|^{-d+1} dx {}&=  M_1^\La C,\\
\sum_{\substack{\xi \in \La^*\backslash \{0\},\\ |\xi| \geq C}} |\xi|^{-d-1}\leq K_2^\La \int_{|x| \geq C} |x|^{-d-1} dx {}&=  M_2^\La  C^{-1},
\end{align*}
we obtain
\begin{align*}
& (a \alpha_f)^{-2}  \sum_{\xi \in b^{-1}\La^*\backslash \{0\}} \int_{S^{d-1}}|\mathcal{F} (g_a)(|\xi|u) |^2 du\\
& =(a \alpha_f)^{-2} \bigg( \sum_{\substack{\xi \in b^{-1}\La^*\backslash \{0\},\\ |\xi| < Ca^{-1}b}} \int_{S^{d-1}}|\mathcal{F} (g_a)(|\xi|u) |^2 du+   \sum_{\substack{\xi \in b^{-1}\La^*\backslash \{0\},\\ |\xi|\geq Ca^{-1}b}} \int_{S^{d-1}}|\mathcal{F} (g_a)(|\xi|u) |^2du\bigg)\\
&\leq  a^{-1} b^{d} R^{d-1} \alpha_f^{-2} \bigg(M_1 C \alpha_{|f|}^2  + M_2C^{-1} \bigg(\int |(f\circ \theta^H)'(t)|dt\bigg)^2\bigg) +O(a^{-\frac{1}{2}}b^d)\\
&\leq a^{-1}b^d M_3 R^{d-1} \alpha_f^{-2}\alpha_{|f|}\int |(f\circ \theta^H)'(t)|dt +O(a^{-\frac{1}{2}}b^d)
\end{align*}
where the last inequality holds for
\begin{align*}
C=\sqrt{\frac{M_2}{M_1}} \alpha_{|f|}^{-1} \int |(f\circ \theta^H)'(t)|dt.
\end{align*}

The case of non-compact support follows similarly by choosing $k=2-\eps \frac{s-d}{s+1}$ in Lemma \ref{generalF}.
\end{proof}

\begin{rem}
When $b=b(a)$, the bound in Theorem \ref{vargeneral} is generally not best possible, see the discussion in Section \ref{ballvar}.
If $\lim_{a\to 0} a^{-1}b =0 $,  the first inequality in Lemma~\ref{generalF} yields a better convergence rate. If $f$ is smooth, further partial integrations in \eqref{intpar} yield even better bounds.
\end{rem}

We conclude this section by stating two refinements of  Lemma \ref{generalF} whose proofs are exactly similar. The first one applies to the situation  where $\La$ is not isotropic, but is varied by a random rotation with smooth density:
\begin{cor}\label{1stecor}
Assume $X$ is $C^3$, $f$ is $C^3$ on $(0,1)$, and $\rho $ satisfies Condition \ref{cond1}. Let $h:S^{d-1} \to \R$ be $C^{d}$. Then there are constants $M_1,M_2 > 0$ depending only on $X$ such that for all $R$ large and $a$ small,
\begin{equation*}
\int_{S^{d-1}}|\mathcal{F}(g_a)(Ru)|^2 h(u)du \leq \begin{dcases}
M_1 C_h \bigg( \int |(f\circ \theta^H)'(t)|dt \bigg)^2\big( R^{-d-1}  + O\big( a R^{-d-\frac{1}{2}}\big)\big)\\
M_2 C_h \bigg( \int |f\circ \theta^H(t)|dt\bigg)^2\big( a^2 R^{-d+1} + O\big( a^3 R^{-d+\frac{3}{2}}\big)\big)
\end{dcases}
\end{equation*}
where 
\begin{equation*}
C_h= \sum_{|\alpha| \leq d} \sup \bigg|\frac{\partial^{|\alpha|}}{\partial u^{\alpha}}h\bigg|.
\end{equation*}
Here the derivatives of $h$ should be understood in a fixed coordinate system on $S^{d-1}$ depending only on $X$. Again, the $O$-term may depend on $f$ and $\rho$.

If $\rho $ satisfies Condition \ref{cond2}, the inequalities hold with $O( a^{-2+(\frac{s-d}{s+1}+1)\eps}R^{-d-2+\eps})$ in the first inequality and $O( a^{(\frac{s-d}{s+1}+1)\eps}R^{-d+\eps})$ in the second where $\frac{2(s+1)}{2s-d+1} <  \eps < 2$.
\end{cor}

\begin{proof}
This follows by a direct copy of the proof of Lemma \ref{generalF}. The only place where $h$ is touched is in the $d$-fold partial integration in the $S^{d-1}$-coordinates. This involves differentiation of $h$ up to $d$ times, giving rise to the constant $C_h$ on the right hand side of the inequality.
\end{proof}

The second refinement gives a convenient approximation of the Fourier coefficients:
\begin{cor}\label{2ndencor}
Assume that $X$ is $C^3$, $f$ is $C^3$ on $(0,1)$, and $\rho $ satisfies Condition \ref{cond1}. Let $h:S^{d-1} \to \R$ be $C^{d}$. Then there is a constant $M>0$ depending only on $X$, $f$, and $\rho$ such that for all $R$ large and $a$ small,
\begin{align*}
\MoveEqLeft \int_{S^{d-1}}\bigg| \mathcal{F}(g_a)(Ru) - a\int_{\partial X} \int_\R  f\circ \theta^H( t) (n \cdot u)^2 e^{-2\pi i R(x+atn) \cdot u} dt \sigma(dx) \bigg|^2 h(u)du \\
&\leq 
M  C_h a R^{-d-\frac{1}{2}} 
\end{align*}
where $C_h$ is as in Corollary \ref{1stecor}.

If $\rho $ satisfies Condition \ref{cond2}, the inequality holds but with $ a^{-2+(\frac{s-d}{s+1}+1)\eps}R^{-d-2+\eps}$ on the right hand side.
\end{cor}

\begin{proof}
Replacing $\nabla_{u} g_a$ by 
\begin{equation*}
\nabla_{u} g_a (x+tn) - a^{-1}(f\circ \theta^H)'(a^{-1} t) n \cdot u 
\end{equation*}
in the proof of Lemma \ref{generalF}  shows that 
\begin{align*}
\MoveEqLeft \int_{S^{d-1}}\bigg| \mathcal{F}(g_a)(Ru) - \frac{a^{-1}}{2\pi i R} \sum_{m=0}^{d-1} \int_{\partial X} \int_\R t^m  (f\circ \theta^H)'(a^{-1} t) (n \cdot u)\\
&\quad \times e^{-2\pi i R(x+tn) \cdot u} dt s_m(x) dx \bigg|^2 h(u)du \\
 &\leq M C_h a R^{-d-\frac{1}{2}}.
\end{align*}
The proof also shows that all terms
\begin{align*}
\int_{S^{d-1}}\bigg|\frac{a^{-1}}{2\pi i R}  \int_{\partial X} \int_\R t^m  (f\circ \theta^H)'(a^{-1}t) (n\cdot u) e^{2\pi i R(x+tn) \cdot u} dt s_m(x) dx \bigg|^2 h(u)du
\end{align*}
with $m>0$ are of order at least $a R^{-d-\frac{1}{2}}$. Partial integration in the $m=0$ term yields the claim.
\end{proof}

\section{More on the case of a ball}\label{ballsec}
This section contains a few extensions of the results in the case where $X$ is a ball. 
The theorems are only stated when $\rho $ satisfies Condition \ref{cond1}, but the case of Condition \ref{cond2} is similar.
Again we first consider the Fourier coefficients:
\begin{lem} \label{ballF}
Let $X=B(R)$.
Suppose $\rho$ satisfies Condition \ref{cond1} and $f$ is $C^2$ on $[\beta, \omega]$, possibly having discontinuities at $\beta, \omega$. Then there are constants $M_1 , M_2>0$ such that for all $R$ sufficiently large 
\begin{align}\nonumber
 \big|\mathcal{F}(f\circ \theta_a^X)( \xi )\big|^2  ={}& 4|\xi|^{-d+1} a^2\bigg( \int_{\varphi(\omega)}^{\varphi(\beta)} f\circ \theta^H(r) \cos(2\pi(R+ar)|\xi| + \nu_d)\\
&\times (R+ar)^{\frac{d-1}{2}} dr\bigg)^2 +O(R^{d-1}a^2|\xi|^{-d-1})\label{jegkedermig}
\end{align}
where the $O$-terms may depend on $f$ and $\rho$. 
In particular,
\begin{equation*}
\big|\mathcal{F}(f\circ \theta_a^X)(\xi)\big|^2 \leq \begin{dcases}
M_1a^2 R^{d-1} |\xi|^{{-d+1}} \bigg( \int |f\circ \theta^{H}(t)| dt  + O(|\xi|^{-1}+a)\bigg)^2\\
 M_2R^{d-1}|\xi|^{{-d-1}} \bigg(|f(\beta)|+|f(\omega)|+\int |(f\circ \theta^{H})'(t)| dt +O(a) \bigg)^2 .
\end{dcases}
\end{equation*}
\end{lem}

\begin{proof}
Noting that $g_a(x)$ only depends on $|x|$ and is supported on $[R-aD,R+aD]$, we rewrite the Fourier integral using \cite[IV, Thm 3.3]{stein}:
\begin{align*}
\mathcal{F}(f\circ \theta_a^X)(\xi) {}&= \int_{\R^d} g_a(z)e^{-2\pi i \xi \cdot z} dz
=2\pi  |\xi|^{-\frac{d-2}{2}} \int_{R-aD}^{R+aD} g_a(r) J_{\frac{d}{2}-1}(2\pi |\xi| r ) r^{\frac{d}{2}}dr
\end{align*}
where $J_{\frac{d}{2}-1}$ is the Bessel function of the first kind.
By \cite[IV, Lem. 3.11]{stein}, 
\begin{equation*}
J_{\frac{d}{2}-1}(x)=\sqrt{\tfrac{2}{\pi x}}\cos(x+\nu_d) + O(x^{-\frac{3}{2}})
\end{equation*}
 where $\nu_d=-\frac{d-1}{4}\pi$, so since
\begin{equation*}
|\xi|^{-\frac{d+1}{2}}
\int_{R-aD}^{R+aD} g_a( r)r^{\frac{d-3}{2}} dr \leq   M_1 a R^{\frac{d-3}{2}}\sup |f||\xi|^{-\frac{d+1}{2}},
\end{equation*}
it is enough to consider
\begin{equation*}
2 |\xi|^{-\frac{d-1}{2}}  \int_{R-aD}^{R+aD} g_a( r) \cos(2\pi |\xi| r + \nu_d) r^{\frac{d-1}{2}} dr.
\end{equation*}

Observe that by Lemma \ref{diff}
\begin{align*}
\MoveEqLeft \bigg|\frac{d}{dr}f\circ \theta_a^X( r) - \frac{d}{dr}f\circ \theta^H( a^{-1}(r-R)) \bigg| \leq  M_2  \mathds{1}_{A_1(a)}(r) + 2M_3a^{-1}\sup |f'| \mathds{1}_{A_2(a)}(r)
\end{align*}
where
\begin{align*}
A_1(a) &= \{ r \in [R-aD, R+aD] \mid \theta_a^X( r),\theta^H( a^{-1}(r-R)) \in [\omega,\beta]\},\\
A_2(a) &= [R-aD,R+aD]\backslash ( A_1(a)\cup \{ r \in [R-aD,R+aD] \mid \\
 &\quad \theta_a^X(r),\theta^H( a^{-1}(r-R)) \notin [\omega,\beta]\}).
\end{align*}
By Lemma \ref{tlemma} there is a constant $M_4$ such that
$\Ha^1(A_2(a)) \leq M_4 a^2$.
Hence \eqref{jegkedermig} follows by partial integration and Lemma \ref{tlemma},
\begin{align*}
\MoveEqLeft  |\xi|^{-\frac{d-1}{2}} \int_{R-aD}^{R+aD} (g_a( r)-f\circ \theta^H( a^{-1}(r-R))) \cos(2\pi|\xi| r +\nu_d) r^{\frac{d-1}{2}} dr\\
&\leq  M_5 R^{\frac{d-1}{2}}  a |\xi|^{-\frac{d+1}{2}}.
\end{align*}
%
The first inequality follows from
\begin{align*} 
\MoveEqLeft |\xi|^{-\frac{d-1}{2}}  \int_{R+a\varphi(\omega)}^{R+a\varphi(\beta)}f\circ \theta^H( a^{-1} (r-R))\cos(2\pi |\xi|r + \nu_d ) r^{\frac{d-1}{2}} dr\\
&\leq 
|\xi|^{-\frac{d-1}{2}} a \int_{\varphi(\beta)}^{\varphi(\omega)}|f\circ \theta^H( r)|(ar+R)^{\frac{d-1}{2}} dr\\
&\leq 
|\xi|^{-\frac{d-1}{2}} a  R^{\frac{d-1}{2}}\bigg( \int_{\varphi(\beta)}^{\varphi(\omega)}|f\circ \theta^H( r)|dr + O(a)\bigg).
\end{align*}
The second inequality follows similarly by partial integration. 
%
\end{proof}

As in Section \ref{X} we obtain:

\begin{thm} \label{ballvar}
Assume that $\rho$ satisfies Condition \ref{cond1} and $f$ is $C^2$ on $[\beta,\omega]$ possibly with discontinuities at $\beta, \omega$.
Then there is a constant $M>0 $ depending only on $\La$ such that for all $R$ large and $a,b$ small 
\begin{equation*}
\Var(\hat{S}^{a,b}(f)(B(R))) \leq a^{-1}b^{d} M R^{d-1} \frac{\alpha_{|f|}}{\alpha_f^2} \bigg(|f(\beta)|+|f(\omega)|+ \int |(f\circ \theta^H)'(t)|dt +O(a)\bigg)
\end{equation*}
where the $O $-term may depend on $f$ and $\rho$.
\end{thm}

For the rest of this section, we assume that $b$ is a function of $a$. 
In the special case $b=a$ where the lattice distance and the PSF are shrinked at the same rate, which is also the case studied in \cite{am4}, we already saw in Section \ref{expsec} that the convergence rate given by Theorem \ref{ballvar} is best possible. We have the following preliminary version of Theorem  \ref{maindet}:
\begin{cor}
Let $X$, $\rho$, and $f$ be as in Lemma \ref{ballF}. If $b=a$,
\begin{align*}
\limsup_a a^{-d-1} |\mathcal{F}(g_a)( a^{-1} \xi )|^2 & =4 |\xi|^{-d+1} R^{d-1} |\mathcal{F} (f\circ \theta^H )(|\xi|)|^2\\
\liminf_a a^{-d-1} |\mathcal{F}(g_a)( a^{-1}\xi )|^2 & = 0.
\end{align*} 
The variance may be decomposed as 
\begin{align*}
\Var(\hat{S}(f)^{a,a}(B(R))) ={}& a^{d-1} 2 \alpha_f \omega_{d}^{-1}(1 +Z(a)) \sum_{\xi \in \La^* \backslash \{0\}} |\xi|^{-d+1} R^{d-1} |\mathcal{F} (f\circ \theta^H )(|\xi|)|^2\\
&+O(a^{d})
\end{align*}
where $Z(a)$ is an oscillating term satisfying $|Z(a)|\leq 1$.
\end{cor}


\begin{proof}
Applying partial integration to \eqref{jegkedermig} yields
\begin{align*}
|\mathcal{F}(g_a)(\xi)|^2 ={}& 4 a^{d+1}|\xi|^{-d+1}
R^{d-1}  \bigg( \int_{\varphi(\omega)}^{\varphi(\beta)} f\circ \theta^H(r) \cos(2\pi(a^{-1}R+r)|\xi| + \nu_d) dr \bigg)^2 \\
&+ O(a^{d+2}|\xi|^{-d-1}).
\end{align*}

Write $h(a)=2\pi a^{-1}R|\xi| + \nu_d$ and 
\begin{equation*}
I=\int_{\varphi(\omega)}^{\varphi(\beta)} f\circ \theta^H(r) e^{2\pi i r|\xi|}dr. 
\end{equation*}
Then
\begin{align*}
 \bigg(2\int_{\varphi(\omega)}^{\varphi(\beta)} f\circ \theta^H(r) \cos(2\pi(a^{-1}R+r)|\xi| + \nu_d) dr\bigg)^2{}& =(e^{ih(a)}I +  + e^{-ih(a)})
\\&=2Re( e^{2ih(a)}I^2) + 2|I|^2.
\end{align*}
This takes it maximum for $e^{ih(a)}=\frac{\bar{I}}{|I|}$ and its minimum for $e^{ih(a)}=-\frac{\bar{I}}{|I|}$. 
In the last expression, the $2Re( e^{2ih(a)}I^2)$-terms and the $O(a^{d+2}|\xi|^{-d-1})$-terms form $Z(a)$. 
\end{proof}

The convergence rate is typically faster. When $b$ converges faster to zero than $a$, we easily obtain the following improvement:
\begin{cor}
Let $X$, $\rho$, and $f$ be as in Lemma \ref{ballF}. For $\lim_{a\to 0 } a^{-1}b =0$,
\begin{equation}\label{varrate}
\Var (\hat{S}(f)^{a,b}(B(R))) \in O(a^{-2}b^{d+1}).
\end{equation}
Moreover,
\begin{align*}
 |\mathcal{F}(g_a)( b^{-1} \xi )|^2  = {}& b^{d+1}|\xi|^{-d-1} R^{d-1}{\tfrac{1}{\pi^2}}\Big(f(\beta)\sin(2\pi b^{-1}(R+a\varphi(\beta))|\xi| + \nu_d)\\
&- f(\omega)\sin(2\pi b^{-1}(R+a\varphi(\omega))|\xi| + \nu_d) + O(a+a^{-1}b|\xi|^{-1})\Big)^2.
\end{align*} 
\end{cor}

The latter statement shows that the convergence rate heavily depends on the smoothness assumptions.
If $f(\beta)\neq f(\omega)$, then $\limsup_a b^{-d-1}|\mathcal{F}(g_a)( b^{-1} \xi )|^2 >0$, so \eqref{varrate} is best possible.
On the other hand, if $f $ and $\rho$ are sufficiently often differentiable, further partial integrations using the asymptotic expansion \cite[Chap. 7.21 (1)]{watson} yield even better convergence rates.

\begin{proof}
The second inequality in Theorem \ref{ballF} immediately yields the first claim.

Applying partial integration to \eqref{jegkedermig}
yields
\begin{align*}
\MoveEqLeft |\mathcal{F}(g_a)( b^{-1} \xi )|^2= b^{d+1}|\xi|^{-d-1} R^{{d-1}}{\frac{1}{\pi^2}} \bigg(f(\beta) \sin(2\pi(R+a\varphi(\beta))b^{-1}|\xi| + \nu_d)\\
&-f(\omega) \sin(2\pi(R+a\varphi(\omega))b^{-1}|\xi| + \nu_d)\\
& - \int_{\varphi(\omega)}^{\varphi(\beta)} (f\circ \theta^H )'(r) \sin(2\pi(R+ar)b^{-1}|\xi| + \nu_d)dr\bigg)^2\\
& +O(R^{d-1}ab^{d+1}|\xi|^{-d-1}).
\end{align*}
Another partial integration shows that the latter integral is of order $O(a^{-1}b|\xi|^{-1})$.
\end{proof}

Finally, when $b$ converges slowly, a bound on the convergence rate is given by:
\begin{cor}
Let $X$, $\rho$, and $f$ be as in Lemma \ref{ballF}.
When $\lim_{a\to 0} ab^{-1} = 0$ and $b(a)$ is continuous with $\lim_{a\to 0} b = 0$,
\begin{align}\label{jul}
&\limsup_a a^{-2}b^{-d+1} |\mathcal{F}(g_a)(|\xi|)|^2 = 4 |\xi|^{-d+1} R^{d-1} \alpha_f^2 >0,\\
&\liminf_a a^{-2}b^{-d+1} |\mathcal{F}(g_a)(|\xi|)|^2 =0.\nonumber
\end{align}
\end{cor}

Note that the bounds \eqref{jul} are not summable, so it is not implied that the variance is $O(b^{d-1})$.

\begin{proof}
Again we have \eqref{jegkedermig}.
Using the addition formulas, write
\begin{align*}
\int_{\varphi(\omega)}^{\varphi(\beta)}{}& f\circ \theta^H(r) \cos(2\pi |\xi|{b^{-1}}(R+ar) + \nu_d)dr\\
={}&\cos(2\pi |\xi|b^{-1}R + \nu_d) \int_{\varphi(\omega)}^{\varphi(\beta)}  f\circ \theta^H(r)  \cos(2\pi |\xi|b^{-1}ar)dr \\
&- \sin(2\pi |\xi|b^{-1}R + \nu_d) \int_{\varphi(\omega)}^{\varphi(\beta)} f\circ \theta^H(r) \sin(2\pi |\xi|b^{-1}ar)dr.
\end{align*}
Clearly, the latter term is $O(ab^{-1})$, while 
\begin{equation*}
\lim_{a\to 0}  \int_{\varphi(\omega)}^{\varphi(\beta)}  f\circ \theta^H(r)  \cos(2\pi |\xi|b^{-1}ar)dr = \int_{\varphi(\omega)}^{\varphi(\beta)}  f\circ \theta^H(r) dr
\end{equation*}
and $\cos^2(2\pi |\xi|b^{-1}R + \nu_d)$ oscillates between $0$ and 1, which yields the result.
\end{proof}

\section{Asymptotic variance formulas}\label{expsec}
In this section, we give some explicit formulas for the variance in the special case $a=b$.


Assume that $X$ is randomly rotated with density $h: \SO(d) \to [0,\infty)$ or equivalently that the lattice $\La$ is randomly rotated with density $Q\mapsto h(Q^{-1})$. 
Then for $v \in S^{d-1}$, $Q^{-1}v \in S^{d-1}$ is random with a density $u \mapsto h_v(u)$ that is $C^\infty$ in $(v,u)\in S^{d-1}\times S^{d-1}$. For $\xi \in \R^{d}\backslash \{0\}$, we let $h_\xi(u)=h_{\xi/|\xi|}(u)$. Then the variance is given by
\begin{align*}
\Var(\hat{S}(f)^{a,a}(X))=a^{-2} \alpha_f^{-2} \sum_{\xi \in \La^*\backslash \{0\}} \int_{S^{d-1}}|\mathcal{F}(g_a)(a^{-1}|\xi| u)|^2 h_\xi(u) du.
\end{align*}
The sum converges by Corollary \ref{1stecor}. We get the following asymptotic formula for the variance:

\begin{thm}\label{varapprox}
Assume that $X$ is a $C^{3}$ manifold, $f$ is $C^3$ on $(0,1)$, and $\rho $ satisfies Condition \ref{cond1}. Let $h: \SO(d)\to \R$ be $C^\infty$. Then
\begin{align*}
\Var(\hat{S}(f)^{a,a}(X))=  {}& \alpha_f^{-2} \sum_{\xi \in \La^*\backslash \{0\}}\int_{S^{d-1}} \bigg| \int_{\partial X} \int_\R  f\circ \theta^H( t) (n\cdot u)^2\\
&\times e^{2\pi i |\xi|(a^{-1}x+tn) \cdot u} dt \sigma( dx) \bigg|^2 h_\xi (u) du\\
& +\sup_{\xi \in \La^*\backslash \{0\}} \{C_{h_\xi}\}O( a^{d-\frac{3}{4}}).
\end{align*}
Here $C_{h_\xi}$ is as in Corollary \ref{1stecor} and the $O$-term depends only on $X$, $\rho$, and $f$.

If Condition \ref{cond2} is satisfied with $s>2d+1$, this holds but with $O(a^{d-1+\eps})$ for some small $\eps>0$.
\end{thm}

\begin{proof}
Corollary \ref{1stecor} and \ref{2ndencor} show that
\begin{align*}
\MoveEqLeft \bigg| a^2\int_{S^{d-1}}\bigg|\int_{\partial X} \int_\R  f\circ \theta^H( t) (n \cdot u)^2  e^{2\pi i |\xi|(a^{-1}x+tn) \cdot u} dt \sigma(dx) \bigg|^2 h_\xi (u)du \\
&-\int_{S^{d-1}}\big| \mathcal{F}(g_a)(a^{-1}|\xi|u)\big|^2 h_\xi(u) du\bigg| \leq 
M  C_{h_\xi} a^{d+\frac{5}{4}}|\xi|^{-d-\frac{3}{4}}. 
\end{align*}
Since $\sup_{\xi \in \La^*\backslash 0} C_{h_\xi} $ is finite, the claim follows. 

The case of Condition \ref{cond2} follows by choosing $\eps$ close to 2 in Corollary \ref{2ndencor}.
\end{proof}

\subsection{The case of a convex set}
We now restrict to the special case where $X$ is a smooth compact convex set with nowhere vanishing Gaussian curvature $K$. In this case, the normal map $n: \partial X \to S^{d-1}$ is a diffeomorphism with inverse $x: S^{d-1} \to \partial X$.

We shall obtain the following description of the variance:
\begin{thm}\label{maindet}
Assume that $X$ is a smooth compact convex set with nowhere vanishing Gaussian curvature, $f$ is $C^3$ on $(0,1)$, $h:\SO(d) \to S^{d-1}$ is smooth, and $\rho$ satisfies Condition \ref{cond1} or \ref{cond2} with $s>2d+1$. Then
\begin{align*}
 \Var(\hat{S}(f)^{a,a}(X))={}& 2a^{d-1} \alpha_f^{-2} \bigg(\sum_{\xi \in \La^*\backslash \{0\}}\big| \mathcal{F}(f\circ \theta^H)(|\xi|) \big|^2|\xi|^{-d+1} \int_{\partial X}  h_\xi(n(x)) \sigma(dx)\\
&+ Z(a)\bigg) 
\end{align*}
where $Z(a)$ is in general an oscillating term satisfying
\begin{align*}
\MoveEqLeft \limsup_a \pm Z(a) \\
&\leq  \sum_{\xi \in \La^*\backslash \{0\}}\big| \mathcal{F}(f\circ \theta^H)(|\xi|) \big|^2|\xi|^{-d+1} \int_{S^{d-1}} (K(x(u))K(x(-u)))^{-\frac{1}{2}} h_\xi(u) du.
\end{align*}
\end{thm}

Using the inequality $2k_1k_2 \leq k_1^{2}+ k_2^{2}$, we obtain the following corollary in the isotropic case:
\begin{cor}\label{detcor}
Assume $X$ is a smooth compact convex set with nowhere vanishing Gaussian curvature, $f$ is $C^3$ on $(0,1)$, $\rho$ satisfies Condition \ref{cond1} or \ref{cond2} with $s>2d+1$, and $\La$ is isotropic. Then
\begin{align*}
\MoveEqLeft \Var(\hat{S}(f)^{a,a}(X))= 2a^{d-1} \omega_{d}^{-1}\alpha_f^{-2}S(X)\bigg(\sum_{\xi \in \La^*\backslash \{0\}}\big| \mathcal{F}(f\circ \theta^H)(|\xi|) \big|^2|\xi|^{-d+1} +Z(a)\bigg)
\end{align*}
where $Z(a)$ is in general an oscillating term satisfying
\begin{align*}
\limsup_a \pm Z(a) \leq  \sum_{\xi \in \La^*\backslash \{0\}} \big| \mathcal{F}(f\circ \theta^H )(|\xi|) \big|^2|\xi|^{-d+1}.
\end{align*}
\end{cor}

The summands in Theorem \ref{varapprox} are described by \cite[Thm. 7.7.14]{lars}:
\begin{prop}\label{propo}
Assume $X$ is a smooth compact convex set with nowhere vanishing Gaussian curvature, $f^H $ is continuous with compact support,  and $h:SO(d) \to S^{d-1}$ is smooth. For $u\in S^{d-1}$ and $R >0$ given, there is a constant $C>0$ such that for all $a<1$,
\begin{align*}
\MoveEqLeft \bigg| a^{-\frac{d-1}{2}}R^{\frac{d-1}{2}} \int_{\partial X} \int_\R  f^H( t) (n\cdot u)^2 e^{2\pi i R(a^{-1}x+ tn) \cdot u} dt \sigma( dx) \\
   -{}&  \sum_{\epsilon = \pm 1} \mathcal{F}(f^H)(-\epsilon R) K(x(\epsilon u))^{-\frac{1}{2}}e^{2\pi i Ra^{-1} x(\epsilon u) \cdot u -
	\epsilon i\pi (d-1)/4}\bigg| \leq C a.
\end{align*}
\end{prop}

In order to integrate with respect to $u$, we need the constant on the right hand side to be independent of $u$: 

\begin{prop}\label{Fourierdet}
The constant $C$ in Proposition \ref{propo} can be chosen independently of $u\in S^{d-1}$.
\end{prop}
Note that the constant is not guaranteed to be uniform in $R$. 
To show Proposition~\ref{Fourierdet}, we repeat the proof of \cite[Thm. 7.7.14]{lars} with a bit more care, see this reference for details.
The proof is based on the following lemma, which is stated in \cite[Thm. 7.7.5]{lars}:
\begin{lem}\label{larslemma}
Let $K\subseteq \R^d$ compact and $U$ an open neighborhood of $K$. Let $v$  be $C^{2k}$ supported on $K$, $\phi $ be real and $ C^{3k+1}$ on $U$ with $\nabla \phi(x_0) = 0$, $\det H(\phi)(x_0) \neq 0$ and $\nabla \phi \neq 0$ on $K\backslash \{x_0\}$. Then for all $\tau>0$,
\begin{align*}
\MoveEqLeft \bigg|\int_{\R^d} v(x)e^{2\pi i\tau\phi(x)}dx - e^{2\pi i\tau\phi(x_0)}e^{i\pi \frac{\sigma}{4} }|\det(\tau H(\phi)(x_0))|^{-\frac{1}{2}}v(x_0) \sum_{l < k}\tau^{-l}L_l v\bigg| \\
&\leq C\tau^{-k} \sum_{|\alpha|\leq 2k} \sup \bigg|\frac{\partial^{|\alpha|}}{\partial x^\alpha } v\bigg|.
\end{align*}
Here $H(\phi)$ denotes the Hessian matrix of $\phi$, 
 $\sigma$ is the signature of $H(\phi)(x_0)$, and $L_l$ is a differential operator of order $2l$ with coefficients that are rational functions in the derivatives of $\phi$ up to order $2l+2$  at $x_0$, involving only a power of $\det( H(\phi)(x_0))$ in the denominator. In particular, $L_0$ is evaluation at $x_0$.

The constant $C$ is uniform for $\phi$ belonging to a bounded subset of $C^{3k+1}$ and $|x-x_0|/|\nabla \phi(x)|$ uniformly bounded.
\end{lem}

\begin{proof}[Proof of Proposition \ref{Fourierdet}]
Choose a partition of $ \partial X$ into open sets $X_j$ such that for each, there is a vector $\xi_j$ with $\xi_j \cdot n(x)\geq \cos (\frac{\pi}{8})$ for all $x\in X_j$. Let $\varphi_j$ be a smooth partition of unity with respect to this covering. Let $\psi_{ij}$ be a smooth partition of unity on $S^{d-1}$ such that $|u\cdot \xi_j |\leq \cos (\frac{\pi}{4}) $ on all of $\supp \psi_{1j}$ and $(-1)^{i} u\cdot \xi_j \geq \cos (\frac{3\pi}{8})$ on $\supp \psi_{ij}$ for $i=2,3$. 

Then we must consider
\begin{equation*}
\sum_j \sum_{i=1,2,3} \int_{X_j} \int_\R f^H (t) (n\cdot u)^2 e^{2\pi i R(a^{-1}x +tn) \cdot u}dt  \varphi_j(x) \psi_{ij}(u) \sigma (dx).
\end{equation*}

On $\supp (\varphi_j \psi_{1j})$, we parametrize $X_j$ as the graph over a plane containing $u$ by rotating $\xi_j^\perp$ as in the proof of Lemma \ref{generalF} and apply partial integration to
\begin{equation*}
\int_{\R^{d-1}} \int_\R f^H (t) (n(y)\cdot u)^2e^{2\pi iR (a^{-1}y + tn(y) )\cdot u}dt \det (J(y,u)) \varphi_j(y)\psi_{1j}(u)dy
\end{equation*}
a large number of times to show that the contribution from the integral is small enough to be ignored.  Here $y$ are the local coordinates and $J$ is a Jacobian depending smoothly on $u$ and $y$.

On $\supp \psi_{ij}$, $i>1$, we parametrize $X_j$ as the graph over $\xi_j^\perp$ and consider
\begin{equation*}
\int_{\R^{d-1}} \int_\R f^H (t) (n(y)\cdot u)^2e^{2\pi iR (a^{-1}x(y) + tn(y) ) \cdot u}dt \det (J(y)) \varphi_j(y)\psi_{ij}(u)dy.
\end{equation*}
This corresponds to Lemma \ref{larslemma} with $\tau =a^{-1}$, $\phi(y)=R{x(y)} \cdot u $, and 
\begin{equation*}
v(y)= \int_\R f^H(t) e^{2\pi iR t n(y) \cdot u}dt (n(y)\cdot u)^2 \det (J(y))\varphi_j(y)\psi_{ij}(u).
\end{equation*}
Here $\nabla \phi(y_0) = 0$ if and only if $n(y_0) $ is parallel to $u$, i.e.\ $y_0=x((-1)^{i}u)$ and in this case $|\det H(\phi)(y_0)| = R^{d-1} K(x(y_0))\det (J(y_0))^2$ which is bounded from above and by a strictly positive constant from below on $\supp \psi_{ij} $ by the curvature assumption. 

Hence Lemma \ref{larslemma} for $k\geq \frac{d+1}{2}$ shows that
\begin{align}\label{larseq}
\MoveEqLeft \bigg|\int_{\R^{d-1}} v(y)e^{2\pi ia^{-1}\phi(y)}dy - e^{2\pi ia^{-1}  \phi(y_0)}e^{i\pi \frac{\sigma}{4} } |\det(a^{-1}H(\phi)(y_0))|^{-\frac{1}{2}} \sum_{l < k} a^l  L_l v \bigg| \\\nonumber
&\leq C a^k  \sum_{|\alpha|\leq 2k} \sup |D^\alpha v|.
\end{align}
Note that for $u\in \supp \psi_{ij}$, $\phi$ stays in a bounded subset of $C^{3k+1}$ and all derivatives of $v$ remain bounded.
Moreover, $|y-y_0|/|\nabla \phi(y)|$ is uniformly bounded since by Taylor's formula
\begin{align*}
\nabla \phi (y) = \nabla \phi (y) -\nabla \phi (y_0) = H( \phi )(y_0)(y-y_0) + O(|y-y_0|^2)
\end{align*}
where the O-term only depends on the third order partial derivatives of $\phi $ and hence is uniform in $u$. Thus
\begin{align*}
|y-y_0|\leq  | H(\phi)(y_0)^{-1}| (|\nabla \phi (y)| + M|y-y_0|^2 ) 
\end{align*}
provides the uniform bound, since 
\begin{align*}
y_0 \mapsto| H(\phi)(y_0)^{-1}|=|J(y_0)^{-1}|^2 \max\{k_i^{-1}(x(\pm u)),i=1,\dots,d-1\},
\end{align*}
 where $k_i$ are the principal curvatures, is uniformly bounded. Hence $C$ can be chosen uniformly.

Moving the $l>0$ terms to the right hand side of \eqref{larseq} and identifying the $l=0$ term yields the claim.
\end{proof}

\begin{proof}[Proof of Theorem \ref{maindet}]
Write $f^H = f\circ \theta^H$. By Theorem \ref{varapprox}, the variance is asymptotically given by
\begin{align*}
\alpha_f^{-2} \sum_{\xi \in \La^*\backslash \{0\}} a^{-d+1} \int_{S^{d-1}}  \bigg| \int_{\partial X} \int_\R  f^H( t) (n\cdot u)^2 e^{2\pi i |\xi|(a^{-1}x+tn) \cdot u} dt \sigma( dx) \bigg|^2 h_\xi (u) du.
\end{align*}
By Corollary \ref{1stecor} and \ref{2ndencor} this converges uniformly when $a\to 0$. Hence
\begin{align*}
& \limsup_a  \sum_{\xi \in \La^*\backslash \{0\}}a^{-d+1}\int_{S^{d-1}} \bigg| \int_{\partial X} \int_\R  f^H( t) (n\cdot u)^2 e^{2\pi i |\xi|(a^{-1}x+tn) \cdot u} dt \sigma( dx) \bigg|^2 h_\xi (u) du\\
&\leq 
\sum_{\xi \in \La^*\backslash \{0\}}\limsup_a \int_{S^{d-1}} a^{-d+1} \bigg| \int_{\partial X} \int_\R  f^H( t) (n\cdot u)^2 e^{2\pi i |\xi|(a^{-1}x+tn) \cdot u} dt \sigma( dx) \bigg|^2 h_\xi (u) du .
\end{align*}

From Proposition \ref{Fourierdet} we get
\begin{align*}
\MoveEqLeft \bigg| a^{-\frac{d-1}{2}}\int_{\partial X} \int_\R  f^H( t) (n\cdot u)^2 e^{2\pi i |\xi|(a^{-1}x+tn) \cdot u} dt \sigma( dx) \\  
&- |\xi|^{-\frac{d-1}{2}} \sum_{\epsilon = \pm 1}\mathcal{F}(f^H)(-\epsilon|\xi|) K(x(\epsilon u))^{-\frac{1}{2}}e^{2\pi i |\xi| a^{-1}x(\epsilon u) \cdot u -\epsilon i\pi (d-1)/4}\bigg| \leq C a  
\end{align*}
where $C$ is uniform in $u$, so that
\begin{align*}
\MoveEqLeft \limsup_a  \int_{S^{d-1}} a^{-d+1}\bigg| \int_{\partial X} \int_\R  f^H( t) (n\cdot u)^2 e^{2\pi i |\xi|(a^{-1}x+tn) \cdot u} dt \sigma( dx)\bigg|^2 h_\xi(u)du \\
&= \limsup_a  |\xi|^{-d+1}  \int_{S^{d-1}}\bigg| \sum_{\epsilon = \pm 1}\mathcal{F}(f^H)(-\epsilon|\xi|) K(x(\epsilon u))^{-\frac{1}{2}}\\
&\quad \times e^{2\pi i|\xi| a^{-1} x(\epsilon u) \cdot u - \epsilon i\pi (d-1)/4}\bigg|^2 h_\xi(u)du.
\end{align*}
Writing the latter sum out, noting that $h_{-\xi}(u)=h_\xi(-u)$, and using that
\begin{equation*}
\int_{S^{d-1}} K(x(u))^{-1}h(u) du = \int_{\partial X} h(n(x)) \sigma(dx),
\end{equation*}
see \cite[Section 2.5]{schneider}, yields the claim.
\end{proof}

\subsection{Random sets}
Following the idea of \cite{kieu}, we now turn to the situation  where we observe a random set $sQX$ where $Q\in \SO(d)$ is a random rotation and $s>0$ a random scaling parameter. We assume $(Q,s)$ has a joint density $h(Q,s)$ that is smooth and compactly supported in $\SO(d)\times (0, \infty)$. 

In this case we try to estimate the mean surface area 
\begin{equation*}
ES(sQX) = S(X) \int_{SO(d)\times \R} s^{d-1} h(Q,s) d(Q,s).
\end{equation*}
When $\partial X$ is smooth, it follows from the proof of Lemma \ref{diff} and the Weyl tube formula that
\begin{equation*}
\bigg| a^{-1} \int_{\R^d} f \circ \theta^{X}_a (x) dx - \alpha_f S(X)\bigg| \leq MS(X)a^\eps
\end{equation*}
where $\eps > 0$ and $M$ only depends on $X$ through  an upper bound on the principal curvatures of $X$. Hence uniform convergence 
and  the assumption  that $s$ is bounded from below and above show that $\hat{S}(f)^{a,a}$ is asymptotically unbiased:
\begin{equation*}
\lim_{a\to 0} E \hat{S}(f)^{a,a}(sQX) = ES(sQX).
\end{equation*}

To describe the variance, observe first that
\begin{equation*} 
\mathcal{F}(f\circ \theta_a^{sQ X}(\xi)) = s^d \mathcal{F}(g_{as^{-1}})(sQ^{-1 } \xi)
\end{equation*}
such that
\begin{align*} 
E\big|\mathcal{F}(f\circ \theta_a^{sQ X}(\xi))\big|^2 = \int_{SO(d)} \int_\R s^{2d} |\mathcal{F}(g_{as^{-1}})(a^{-1}sQ^{-1 } \xi)|^2 h(Q,s ) ds dQ\\
= \int_\R \int_{S^{d-1}}  s^{2d} |\mathcal{F}(g_{as^{-1}})(a^{-1} s|\xi|u) |^2 h_\xi(u,s)  du ds .
\end{align*}

In this case, the calculations in the previous section carry over since the partial derivatives of $h_\xi$ are bounded in $(u,s,\xi)$ simultaneously and the curvature of $sQX$ is bounded from both below and above.
We note only the special case where $X$ is isotropic so that $h(Q,s)=h(s)$. Then the variance converges for $a\to 0$:
\begin{thm}\label{randomdet}
Assume $X$ is a smooth compact convex set with nowhere vanishing Gaussian curvature. Consider the random set $sQX$ where $Q\in SO(d)$ is uniform random and $s\in (0,\infty)$ is random with smooth compactly supported density $h$. Suppose $f$ is $C^3$ and  $\rho$ satisfies Condition \ref{cond1} or \ref{cond2} with $s>2d+1$. 

 Then $\hat{S}(f)^{a,a}(sQX)$ is asymptotically unbiased and
\begin{align}\label{limvar}
\MoveEqLeft \lim_{a \to 0} a^{-d+1}\Var(\hat{S}(f)^{a,a}(sQX))\\
&= 2 \omega_{d}^{-1}\alpha_f^{-2}ES(sQX)\sum_{\xi \in \La^*\backslash \{0\}} \big| \mathcal{F}(f\circ \theta^H)(|\xi|) \big|^2|\xi|^{-d+1}.\nonumber
\end{align}
\end{thm}

\begin{proof}
The proof goes as the proof of Theorem \ref{maindet}. Dominated convergence and Theorem \ref{varapprox} and Corollary \ref{1stecor} and \ref{2ndencor} show that the left hand side of \eqref{limvar} is
\begin{align*}
& \lim_{a\to 0} a^{-d+1} \omega_{d}^{-1}\alpha_f^{-2}\sum_{\xi \in \La^*\backslash \{0\}}E \bigg|s^{d-1} \int_{\partial X} \int_\R  f^H( t) (n\cdot u)^2 e^{2\pi i |\xi|(sa^{-1}x+tn) \cdot u} dt \sigma( dx) \bigg|^2 \\
 &=
 \sum_{\xi \in \La^*\backslash \{0\}}\lim_{a\to 0}  a^{-d+1} \omega_{d}^{-1}\alpha_f^{-2}E\bigg| s^{d-1} \int_{\partial X} \int_\R  f^H( t) (n\cdot u)^2 e^{2\pi i |\xi|(sa^{-1} x+tn) \cdot u} dt \sigma( dx) \bigg|^2 
\end{align*}
if the latter limits exist. But since the principal curvatures of $sX$ are bounded from above and below, the proof of Proposition \ref{Fourierdet} carries over to show that the constant can be chosen independently of $s$ and hence:
\begin{align*}
\MoveEqLeft \lim_{a\to 0} a^{-d+1} \int_{S^{d-1}} \int_\R s^{2d-2}h(s) \bigg| \int_{\partial X} \int_\R  f^H( t) (n\cdot u)^2 e^{2\pi i |\xi|(sa^{-1}x+tn) \cdot u} dt \sigma( dx)\bigg|^2  ds du\\
 & =  2ES(sQX)\big|\mathcal{F}(f^H)(|\xi|)\big|^2 |\xi|^{-d+1} + \lim_{a\to 0}  2 \textrm{Re} \bigg(\mathcal{F}(f^H)(-|\xi|)^2 |\xi|^{-d+1}e^{-i\pi\frac{d-1}{2}}  \\
&\quad \times \int_{S^{d-1}} (K(x( u))K(x( -u)))^{-\frac{1}{2}} \int_{\R} s^{d-1} h(s)  e^{2\pi i|\xi|s a^{-1}  (x( u) -x(- u))\cdot u } dsdu\bigg)\\
& = 2ES(sQX)\big|\mathcal{F}(f^H)(|\xi|)\big|^2 |\xi|^{-d+1} .
\end{align*}
The last equality follows from dominated convergence, using that $\mathcal{F}(s^{d-1}h(s))(t) \to 0$ for $|t|\to \infty$ and the fact that 
\begin{equation*}
(x( u) -x(- u))\cdot u \neq 0
\end{equation*}
when $X$ has non-empty interior.
\end{proof}

\section{Discussion and open questions}
This paper shows that the asymptotic variance of order at most $O(a^{-1}b^d)$, implying that the variance is relatively well behaved asymptotically. 
It follows that the variance increases when $a$ becomes small and decreases with $b$ as one would expect. However, it is interesting that 
the dependence on $b$ is much stronger. In particular, $\lim_{a \to 0} \Var(\hat{S}(f)^{a,b}(X))= 0 $ whenever $b \in o(a^{\frac{1}{d}})$. For $a=b$ the variance is of order $O(a^{d-1})$. For comparison, the order for volume estimators was $O(a^{d+1})$.

The bounding constant in Theorem \ref{vargeneral} is not claimed to be best possible. For $a=b$ it seems to come from a bound on $\mathcal{F}(f\circ \theta^H)(|\xi|)$. However, keeping \eqref{hurtig} in mind, it still seems informative to investigate which weight function minimizes it. Since $\hat{S}(f)$ is normalized by the factor $\alpha_f$, scaling $f$ does not change the algorithm. Hence we may assume $\sup f =1$. Clearly we want $f$ to be positive and as close to $\mathds{1}_{(0,1)}$ as possible to minimize $\alpha_{|f|}\alpha_f^{-2}$. Moreover, $\int_\R |(f\circ \theta^H)'(t)|dt$ is minimal when $f$ has a single local maximum. All this suggests using (a $C^2$ approximation of) an indicator function $f=\mathds{1}_{[\beta, \omega]}$ with $[\beta, \omega] \subseteq (0,1)$ large.

On the other hand, choosing $[\beta, \omega]$ too large will slow down the convergence of the mean, see \cite{am4}. Here it was also suggested to choose $\omega=1-\beta$ in order to ensure that the asymptotic bias is only of order $O(a^2)$.

The results of this paper also provide an estimation formula for the variance that applies to a certain class of convex sets. It seems likely that the set class can be extended, c.f.\ \cite{janacek,kieu}.
In general, the conditions on $X$, $f$, and $\rho$ may not have been squeezed. The main focus has been on obtaining results that hold in the case of a non-compactly supported PSF, since this is the situation most commonly asked for by the applying scientists, see \cite{kothe}. Though unbiasedness results no longer hold, the case of a non-radial PSF is also of interest.

The assumption of the paper has been that the lattice is randomly rotated. The strength, however, of grey-scale images is that they are asymptotically unbiased even when the lattice orientation is fixed. 
The problem in this case seems to be the approximation result in Corollary \ref{2ndencor}. From there, the results would carry over. 

In the isotropic case, there are also asymptotically unbiased estimators based on black-and-white images \cite{am2} and it would be interesting to know whether the above techniques apply in this setting too. 
Asymptotically unbiased estimators for the integrated mean curvature are  known too, but results may be harder to obtain. 



\section*{Acknowledgements}
This research was funded by a grant from the Carlsberg Foundation and hosted by the Institute of Stochastics at Karlsruhe Institute of Technology. The author wishes to thank Markus Kiderlen for helpful advice.


\end{document}